\documentclass[10pt, a4paper]{amsart} 
\usepackage[left=2.50cm, right=2.50cm, top=2.50cm, bottom=2.50cm]{geometry} 
\usepackage{amscd, amsfonts, amsmath, amssymb, amsthm, arydshln, fixmath, graphicx, mathrsfs, overpic, tikz-cd}
\usepackage{mathtools}
\usepackage{hyperref}
\usepackage[all]{xy} 
\usepackage{tikz}
\usepackage{pgfplots}
\usepackage{pgflibraryarrows}
\usepackage{pgflibrarysnakes}
\usepackage{dsfont}
\makeindex

\theoremstyle{definition}
\newtheorem{Unity}{Unity}[section] 
\newtheorem*{Definition*}{Definition} 
\newtheorem{Definition}[Unity]{Definition}

\theoremstyle{plain} 
\newtheorem*{Theorem*}{Theorem}
\newtheorem{Theorem}[Unity]{Theorem}
\newtheorem{Proposition}[Unity]{Proposition}
\newtheorem{Corollary}[Unity]{Corollary}
\newtheorem{Lemma}[Unity]{Lemma}

\theoremstyle{remark} 
\newtheorem*{Remark*}{Remark}
\newtheorem{Remark}[Unity]{Remark}
\newtheorem{Example}[Unity]{Example}

\newcommand{\cO}{\mathcal{O}}

\newcommand{\End}{\mathrm{End}}

\newcommand{\Hom}{\mathrm{Hom}}

\newcommand{\Rep}{\mathrm{Rep}}

\newcommand{\Qcoh}{\mathfrak{Qcoh}}
\newcommand{\Vect}{\mathfrak{Vect}}

\newcommand{\coker}{\mathrm{coker\,}}

\begin{document}

\title{Exact sequences of fundamental group schemes for $G$-torsors}
\author{Lingguang Li}
\address{School of Mathematical Sciences,
Key Laboratory of Intelligent Computing and Applications (Tongji University), Ministry of Education, Shanghai 200092, CHINA}
\email{LiLg@tongji.edu.cn}
\author{Hao Wang}
\address{School of Mathematical Sciences,
Key Laboratory of Intelligent Computing and Applications (Tongji University), Ministry of Education, Shanghai 200092, CHINA}
\email{wanghao0410@tongji.edu.cn}

\begin{abstract}
Let $k$ be a field, $G$ an affine $k$-group scheme, $X$ a connected scheme proper over $k$,
$f:(Y,y)\to(X,x)$ a pointed $G$-torsor, $\mathcal{C}_X$ and $\mathcal{C}_Y $
 the Tannakian categories of vector bundles on $X$ and $Y$, such that $f^*\mathcal{C}_X\subseteq \mathcal{C}_Y$, we give the necessary and sufficient conditions for the exactness of the natural sequence of $k$-group schemes
$1\rightarrow \pi(\mathcal{C}_Y,y)
        \rightarrow \pi(\mathcal{C}_X,x)
        \rightarrow G
        \rightarrow 1.$
This provides a common framework of exact sequences for the $S$-, Nori, $F$-, \'etale fundamental group schemes and their extended variants. Finally, an isogeny of an elliptic curve shows that the relevant preservation property fails for the local, extended local, and unipotent categories.
\end{abstract}
\maketitle

\tableofcontents

\section{Introduction}
Let $k$ be a field and  $f:(Y,y)\to(X,x)$ a pointed torsor
under an affine $k$-group scheme $G$. Suppose that
$\mathcal{C}_X\subset\Vect(X)$ and $\mathcal{C}_Y\subset\Vect(Y)$ are  Tannakian
categories of vector bundles with $f^*\mathcal{C}_X \subseteq \mathcal{C}_Y$. The pullback functor
induces a homomorphism
$\pi(\mathcal{C}_Y,y)\rightarrow\pi(\mathcal{C}_X,x)$.
If the vector bundles associated with finite-dimensional representations of
$G$ belong to $\mathcal{C}_X$, the associated-bundle functor also induces a
homomorphism $\pi(\mathcal{C}_X,x)\to G$. The basic question is when these morphisms form an
exact sequence
\[
1\rightarrow
\pi(\mathcal C_Y,y)
\rightarrow
\pi(\mathcal C_X,x)
\rightarrow
G
\rightarrow 1.
\]

Exact sequences of this form are known in several geometric settings.
Antei--Emsalem proved the Nori exact sequence for pointed quotient torsors under finite group schemes, under the assumption
$H^0(Y,\cO_Y)=k$ \cite[Theorem~4.1]{AnEm11}. In the language of fundamental
gerbes, Antei--Biswas--Emsalem--Tonini--Zhang studied essentially finite
covers of pseudo-proper inflexible algebraic stacks; in particular, their
results recover the Nori and $S$-fundamental short exact sequences for pointed finite torsors \cite[Corollary~I and Theorem~IV]{ABETZ19}.  Over an
algebraically closed field of positive characteristic,
Biswas--Hai--dos Santos obtained exact sequences for finite \'etale Galois
covers of varieties connected by proper chains in the $F$-divided and
essentially finite settings \cite[Propositions~5.3 and~8.1]{BHS21}. 
Adroja--Amrutiya also proved the exact sequence of $EN$-fundamental group schemes of finite \'etale Galois covers \cite[Corollary~4.9]{AdAm25}.
These results suggest that the existence of such exact sequences is determined by a common Tannakian formalism.

Related questions also arise for quotient morphisms.
For example, a complementary result of
Biswas--Hogadi--Parameswaran concerns GIT quotients: for a connected reductive
affine group acting on a smooth projective variety over an
algebraically closed field with respect to an equivariant very ample line
bundle, a nonempty GIT quotient has the same algebraic fundamental group as
the original variety \cite[Theorem~1~(1)]{BHP15}.  Over $\mathbb C$, 
the analogous statement holds for
topological fundamental groups \cite[Theorem~1~(2)]{BHP15}.

The purpose of this paper is to isolate a common Tannakian formalism behind
the torsor exact sequences above. Rather than treating each
fundamental group scheme separately, we separate the formal Tannakian input
from the geometric problem of verifying it for a given category. Following
the standard associated-bundle construction, let
\[
\eta_X^Y:\Rep_k^f(G)\rightarrow\Vect(X),
\quad
V\longmapsto(\cO_Y\otimes_kV)^G
\]
be the associated-bundle functor.  Our main result is the following.

\begin{Theorem}[Theorem~\ref{thm:general-exact-sequence}]
Let $k$ be a field, $G$ an affine group scheme over $k$, $X$ a connected scheme proper over $k$,
$f:Y\rightarrow X$ a $G$-torsor, $\mathcal{C}_X\subset \Vect(X)$ and $\mathcal{C}_Y \subset \Vect(Y)$
 the Tannakian categories on $X$ and $Y$, such that $f^*\mathcal{C}_X\subseteq \mathcal{C}_Y$, $y\in Y(k)$ a rational point lying over $x \in X(k)$. 
Then:
\begin{enumerate}
 \item If $\eta_X^Y\bigl(\Rep^f_k(G)\bigr)\subseteq \mathcal{C}_X$, then the tensor functor $\eta_X^Y:\Rep^f_k(G) \to \mathcal{C}_X$ induces a faithfully flat homomorphism $ 
 \pi(\mathcal{C}_X,x) \to G$.  
\item If $\eta_X^Y\bigl(\Rep^f_k(G)\bigr)\subseteq \mathcal{C}_X$, then $f^*:\mathcal{C}_X\to \mathcal{C}_Y$ is observable if and only if the natural sequence of affine group schemes over $k$ is exact:
\[
\pi(\mathcal{C}_Y, y)\rightarrow \pi(\mathcal{C}_X, x)\rightarrow G \rightarrow 1.
\]
\item $f_*\mathcal{C}_Y \subseteq \mathcal{C}_X$ if and only if $G$ is a finite $k$-group scheme, $\eta_X^Y\bigl(\Rep^f_k(G)\bigr)\subseteq \mathcal{C}_X$ and the natural sequence of affine group schemes over $k$ is exact:
\[
        1\rightarrow \pi(\mathcal{C}_Y,y)
        \rightarrow \pi(\mathcal{C}_X,x)
        \rightarrow G
        \rightarrow 1.
\]  
\end{enumerate}
\end{Theorem}

The three hypotheses in the theorem play different roles. The associated-bundle
condition identifies the quotient by $G$; observability controls exactness at
the middle term; and preservation under pushforward forces the first
morphism to be a closed immersion. The same formalism is stable under
passing to the saturation of a Tannakian category; see
Corollary~\ref{thm:saturation-exact-sequence}. Thus a category and its extension-closed saturation can be treated by the same argument.

We apply the criterion to the $S$-, Nori, extended Nori, $F$-, extended
$F$-, local, extended local, \'etale, extended \'etale, and unipotent
fundamental group schemes.  Nori's fundamental group scheme is defined by essentially
finite bundles \cite{Nor76,Nor82}, while numerically flat bundles define the
$S$-fundamental group scheme \cite{Lan11,Lan12}.  In positive characteristic,
the $F$- and local fundamental group schemes use Frobenius-finite and
Frobenius-trivial bundles \cite{AmBi10,MeSu08}.  We also use the extended Nori and extended local
categories considered in \cite{Ota17,AdAm25}, and generalized them to other extension-closed categories. In the $S$-, Nori and extended Nori cases,
the resulting short exact sequences are already known by the results cited
above. Here they are recovered from the same Tannakian criterion, which also
applies to the corresponding saturated categories and to the additional
cases described below.

\begin{Theorem}[Theorems~\ref{finite-etale-exact}]
Let $k$ be a field, $X$ and $Y$ geometrically reduced connected schemes proper over $k$, $f: Y \to X$ a $G$-torsor with $G$ an affine $k$-group scheme, $y\in Y(k)$ lying over $x \in X(k)$. If $f$ is a finite \'etale Galois covering and $*\in\{F,EF,et,E\acute et\}$,
Then the following natural sequence is exact:
\[
1\rightarrow\pi^*(Y,y)\rightarrow\pi^*(X,x)
\rightarrow G\rightarrow1.
\]
\end{Theorem}

The local, extended local, and unipotent categories also indicate a natural limitation of this method. In positive characteristic, multiplication by a prime $\ell\ne p$ on an elliptic curve is a finite \'etale Galois cover, but
the pushforward of its structure sheaf belongs to none of these three categories; see Example~\ref{elliptic-counterexample}. Thus, even for a finite \'etale torsor, the associated-bundle and pushforward hypotheses in
the structural theorem need not hold for these Tannakian categories.

The proof of the structural theorem is based on fpqc descent. It identifies representations of $G$ with the objects of $\mathcal{C}_X$ that become trivial after pullback to $Y$. The exactness criterion of Esnault--Hai--Sun then reduces
exactness at the middle term to observability. Under the pushforward hypothesis, the adjunction counit gives observability and the subquotient condition required for the first morphism to be a closed immersion. The remaining arguments verify these hypotheses in the concrete Tannakian categories considered above.

Section~2 recalls the Tannakian criteria used in the proof. Section~3 proves the torsor criterion and its saturated form. Section~4 applies the criterion to the fundamental group schemes listed above and gives the elliptic-curve example showing the failure of the preservation hypothesis in the local, extended local, and unipotent cases.

\section{Preliminaries}

\textbf{Notation and Convention:} Let $k$ be a field, $X$ a scheme over $k$, $x\in X(k)$, $G$ an affine group scheme over $k$,  $\Qcoh(X)$ the category of quasi-coherent sheaves on $X$, $\Vect(X)$ the category of vector bundles on $X$. A Tannakian category of vector bundles $\mathcal{C}_X$ on $X$ is an abelian rigid tensor full subcategory of $\Vect(X)$, and a neutral Tannakian category is a Tannakian category $\mathcal{C}_X$ equipped with a fibre functor  $|_x:\mathcal{C}_X\rightarrow\operatorname{Vect}_k, E\mapsto E|_x$. The Tannakian group scheme of the neutral Tannakian category $(\mathcal{C}_X,|_x)$ is denoted by $\pi(\mathcal{C}_X,x)$. 

\begin{Definition}
    Let $k$ be a field, $\mathcal{C},\mathcal{D}$ Tannakian categories over $k$, $\Phi:\mathcal{C}\rightarrow \mathcal{D}$ an exact tensor functor. The functor $\Phi$ is said to be \textit{observable} if for any $E\in\mathcal{C}$ and any subobject $L\hookrightarrow \Phi(E)\in\mathcal{D}$ of rank one, there exists $F\in\mathcal{C}$ and an integer $n>0$ such that $(L^{\otimes n})^{\vee}\hookrightarrow \Phi(F)$ is a subobject in $\mathcal{D}$.
\end{Definition}

\begin{Lemma}[{\cite[Proposition~A.3]{DaEs22}}]\label{observable}
Let $k$ be a field, $\mathcal{C},\mathcal{D}$  Tannakian categories over $k$, $\Phi:\mathcal{C}\rightarrow \mathcal{D}$ an exact tensor functor. 
Then the following conditions are equivalent:
    \begin{enumerate}
        \item The functor $\Phi$ is observable.
        \item For any $E\in\mathcal{C}$ and any subobject $E'\hookrightarrow\Phi(E)\in\mathcal{D}$, there exists $F\in \mathcal{C}$ such that $\Phi(F)\twoheadrightarrow E' \in\mathcal{D}$.
    \end{enumerate}
\end{Lemma}

\begin{Definition}
Let $k$ be a field, $X$ a connected scheme over $k$, $x\in X(k)$, $\mathcal{C}_X$ the Tannakian category over $X$. Define the \textit{saturation category} $\overline{\mathcal{C}}_X$ of $\mathcal{C}_X$ as the full subcategory of $\Vect(X)$ whose objects are those $E$ for which there exists a filtration
\[
0= E_0 \hookrightarrow E_1 \hookrightarrow \cdots\hookrightarrow E_n=E,
\]
such that $E_i/E_{i-1}\in\mathcal{C}_X$ for any $i$.
\end{Definition}
The following proposition has been proved in \cite{LiWang26}, for self-contained we restate the proof here.
\begin{Proposition}\label{sat-tannakian}
Let $k$ be a field, $X$ a connected scheme over $k$, $\mathcal{C}_X$ an abelian rigid tensor full subcategory of $\Vect(X)$. Then
the saturation $\overline{\mathcal{C}}_X$ of $\mathcal{C}_X$ is also an abelian rigid tensor category.
\end{Proposition}

\begin{proof}
For $E,F\in\overline{\mathcal{C}}_X$, choose finite filtrations with subquotients $\operatorname{gr}_i(E)$ and $\operatorname{gr}_j(F)$ in $\mathcal{C}_X$. Then $E^\vee$ admits a filtration with subquotients $\operatorname{gr}_i(E)^\vee\in\mathcal C_X$, and $E\otimes F$ admits a filtration with subquotients $\operatorname{gr}_i(E)\otimes \operatorname{gr}_j(F)\in\mathcal C_X$. It follows that $E^\vee, E\otimes F \in \overline{\mathcal C}_X$. Hence $\overline{\mathcal C}_X$ is a rigid tensor category.

Now we show that $\overline{\mathcal C}_X$ is an abelian category. For a morphism $f\colon E\to F$ in $\overline{{\mathcal C}}_X$, one proves that the kernel and cokernel belong to ${\mathcal C}_X$ by induction on $\operatorname{rk}(E)+\operatorname{rk}(F)$. If $E$ and $F$ are irreducible objects in $\mathcal{C}_X$, then this claim follows immediately since $\mathcal C_X$ is abelian. If $F$ is reducible, take $0\rightarrow F'\xrightarrow{i} F\xrightarrow{j} F''\rightarrow 0$ with $F''$ is irreducible in $\mathcal{C}_X$. If $\mathrm{Im}(j\circ f)=0$, then $f$ factors through a morphism $g:E\to F'$. Then we have $\ker f\cong\ker g$ and an exact sequence $0\rightarrow \operatorname{coker} g\rightarrow \operatorname{coker} f\rightarrow F''\rightarrow 0$ by the snake lemma. By induction $\ker f$ and $\operatorname{coker} f\in\overline{\mathcal{C}}_X$. If $\mathrm{Im}(j\circ f)=F''$, we have the following commutative diagram of vector bundles
    \[
        \begin{tikzcd}
            0\arrow[r]&\ker(j\circ f)\arrow[r]\arrow[d,"u"]&E\arrow[d,"f"]\arrow[r,"{j\circ f}"]&F''\arrow[d,equal]\arrow[r]&0\\
            0\arrow[r]&F'\arrow[r,"i"]&F\arrow[r,"j"]&F''\arrow[r]&0.
        \end{tikzcd}
    \]
It follows that $\ker f\cong \ker u$ and $\operatorname{coker} f\cong \operatorname{coker}u\in\overline{\mathcal{C}}_X$ by induction. If $E$ is reducible, take $0\rightarrow E'\xrightarrow{p} E\xrightarrow{q} E''\rightarrow 0$ with $E'$ is irreducible in $\mathcal{C}_X$. If $\ker(f\circ p)=E'$, then $f$ factors through a morphism $h:E''\to F$. Then we have $\operatorname{coker} f\cong\operatorname{coker}h$ and an exact sequence $0\rightarrow E'\rightarrow \ker f\rightarrow \ker h\rightarrow 0$ by the snake lemma. By induction $\ker f$ and $\operatorname{coker} f\in\overline{\mathcal{C}}_X$. If $\ker(f\circ p)=0$, we have the following commutative diagram of vector bundles
    \[
        \begin{tikzcd}
            0\arrow[r]& E'\arrow[r,"p"]\arrow[d,equal]&E\arrow[d,"f"]\arrow[r,"q"]&E''\arrow[d,"v"]\arrow[r]&0\\
            0\arrow[r]&E'\arrow[r,"{f\circ p}"]&F\arrow[r]&\operatorname{coker}(f\circ p)\arrow[r]&0.
        \end{tikzcd}
    \]
It follows that $\ker f\cong \ker v$ and $\operatorname{coker} f\cong \operatorname{coker}v\in\overline{\mathcal{C}}_X$ by induction.
\end{proof}

Let $k$ be a field, $G$ an affine group scheme over $k$, $X$ a scheme proper over $k$. Denote by $\Rep_k^f(G)$ the finite--dimensional representation category of $G$ and by $\Rep_k(G)$ the representation category of $G$. The regular representation of $G$ on $k[G]$ is given by $(g\cdot f)(h)=f(hg)$ for any $g,h\in G$ and $f\in k[G]$. Given a $G$-torsor $\phi:P\rightarrow X$, we can define an exact tensor functor 
\[
    \eta_X^P:\Rep_k^f(G)\rightarrow \Qcoh(X), \quad V\mapsto (\cO_P\otimes_k V)^G, 
\]
where the action is the diagonal action. Note that for any $V\in\Rep_k(G)$, $V$ is a union of its finite--dimensional $G$-invariant subspaces. Hence for $V\in \Rep_k(G)$, we may define 
\[
\eta_X^P(V):=\varinjlim\eta_X^P(V_\alpha),
\]
where $V_\alpha$ ranges over all finite--dimensional $G$-invariant subspaces of $V$.

\begin{Lemma}[{\cite[Chapter~VIII, Proposition~10.2]{Mil12}}]\label{regularrepresentation}
    Let $k$ be a field, $G$ an affine group scheme over $k$, $(V, \rho)\in \Rep_k^f(G)$. Then $V$ embeds into a finite sum of copies of the regular representation, i.e. there exists an integer $n>0$ and a $G$-equivariant homomorphism $V\hookrightarrow k[G]^{\oplus n}$ where $k[G]$ is the regular representation.
\end{Lemma}

\begin{Theorem}[{\cite[Theorem A.1]{EHS07}} and {\cite[Remark~2.2]{Zha14}} ]\label{Thm1}
Let $L\xrightarrow{q}G\xrightarrow{p} A$ be a sequence of homomorphisms of affine group schemes over a field $k$. It induces a sequence of functors:
\[
\Rep_k^f(A)\xrightarrow{p^*}\Rep_k^f(G)\xrightarrow{q^*}\Rep_k^f(L),
\]
where $\Rep_k^f$ denotes the category of finite--dimensional representations over $k$. Then we have the following:
\begin{enumerate}
    \item[(1)] The group homomorphism $p: G\rightarrow A$ is faithfully flat iff  $p^*\Rep_k^f(A)$ is a full subcategory of $\Rep_k^f(G)$ and closed under taking subobjects.
    \item[(2)] The group homomorphism $q:L\rightarrow G$ is a closed immersion iff any object of $\Rep_k^f(L)$ is a subquotient of an object of the form $q^*(V)$ for some $V\in \Rep_k^f(G)$.
    \item[(3)] If $p$ is faithfully flat, then the sequence $L\xrightarrow{q}G\xrightarrow{p} A$ is exact iff the following conditions are fulfilled:
    \begin{enumerate}
        \item[(a)] For an object $V\in \Rep_k^f(G)$,  $q^* V\in\Rep_k^f(L)$ is trivial iff $V\cong p^*U$ for some $U\in\Rep_k^f(A)$.\label{3a}
        \item[(b)]\label{3b} Let $W_0$ be the maximal trivial subobject of $q^*V$ in $\Rep_k^f(L)$. Then there exists $V_0\hookrightarrow V$ in $\Rep_k^f(G)$ such that $q^*V_0\cong W_0$.
        \item[(c)] For any $W\in\Rep_k^f(G)$ and any quotient $q^*W\twoheadrightarrow W'\in\Rep_k^f(L)$, there exists $V\in \Rep_k^f(G)$  and an embedding $W'\hookrightarrow q^*V \in \Rep_k^f(L)$.
    \end{enumerate}
\end{enumerate}
\end{Theorem}

\section{Tannakian group schemes of torsors}
\begin{Proposition}\label{fully faith}
Let $k$ be a field, $G$ an affine group scheme over $k$, $X$ a connected scheme proper over $k$,
$f:Y\rightarrow X $ a $G$-torsor. Then the tensor functor
\[
        \eta^Y_X:\Rep^f_k(G) \rightarrow \Vect(X),
        \quad
        V\mapsto (\cO_Y\otimes_k V)^G
\]
is fully faithful if and only if $H^0(Y,\cO_Y)=k$.
\end{Proposition}
\begin{proof}
For any $V, W \in \operatorname{Rep}^f_k(G)$, set $A :=H^0(Y, \cO_Y)$.
Since $f:Y\to X$ is a $G$-torsor, the fpqc
descent induces a tensor equivalence
$f^*:\Vect(X)\xrightarrow{\sim}\Vect^G(Y)$; see \cite[Theorem~4.46]{Vis05}  and \cite[\S 2.1]{AOV08}. Hence  
\[
\begin{aligned}
\Hom_{\Vect(X)}\bigl(\eta^Y_X(V),\eta^Y_X(W)\bigr)
&\cong
\Hom_{\Vect^G(Y)}
\bigl(f^*\eta^Y_X(V),f^*\eta^Y_X(W)\bigr)
\cong
\Hom_{\Vect^G(Y)}
\bigl(\cO_Y\otimes_k V,\cO_Y\otimes_k W\bigr)
\\
&\cong
\Hom_{\Vect(Y)}
\bigl(\cO_Y\otimes_k V,\cO_Y\otimes_k W\bigr)^G
\cong
\bigl(A\otimes_k V^\vee\otimes_k W\bigr)^G
\cong
\bigl(A\otimes_k\Hom_k(V,W)\bigr)^G.
\end{aligned}
\]

If $A=k$, then the $G$-action on $k$ is trivial, and
\[
\Hom_{\Vect(X)}(\eta^Y_X(V), \eta^Y_X(W)) \cong \bigl(A\otimes_k \Hom_k(V,W)\bigr)^G \cong \Hom_k(V, W)^G \cong \Hom_G(V, W).
\]
Hence $\eta_X^Y$ is fully faithful.

Conversely, suppose that $\eta_X^Y$ is fully faithful. Fix $a\in A$. Since $A$ is a  $G$-module,  there exists a
finite-dimensional $G$-submodule $U\subset A$ containing $a$.
Since $k\cdot 1_A\subset A$ is a trivial $G$-submodule, after
replacing $U$ by $U+k\cdot 1_A$, we may assume that
$1_A,a\in U$.
Regard $k$ as the trivial $G$-representation. We have the isomorphism
$\Hom_G(U,k) \cong \Hom_{\Vect(X)}\bigl(\eta^Y_X(U),\eta^Y_X(k)\bigr)$ by full faithfulness, and the isomorphism
\[
\begin{aligned}
\Hom_G(U,k) &\cong  \Hom_{\Vect(X)}\bigl(\eta^Y_X(U),\eta^Y_X(k)\bigr)
\cong
\Hom_{\Vect^G(Y)}
\bigl(\cO_Y\otimes_k U,\cO_Y\otimes_k k\bigr)
\\
&\cong
\Hom_{\Vect(Y)}
\bigl(\cO_Y\otimes_k U,\cO_Y\otimes_k k\bigr)^G
\cong
\bigl(A\otimes_k U^\vee\otimes_k k\bigr)^G \cong \Hom_G(U, A).
\end{aligned}
\]
is given by $\varphi\mapsto \bigl(u\mapsto \varphi(u)1_A\bigr)$. Let $(\iota: U \rightarrow A) \in \Hom_G(U, A)$ be the inclusion.
By the surjectivity, there exists $\varphi\in\Hom_G(U,k)$ such that
$\iota(u)=\varphi(u)1_A$
for every $u\in U$, hence $U\subseteq k\cdot 1_A$.
Since $1_A\in U$, we obtain $U=k\cdot 1_A$, and therefore
$a\in k\cdot 1_A$. Since $a\in A$ was arbitrary, we have  $A=k$.
\end{proof}

\begin{Proposition}\label{regular-representation-inclusion}
Let $k$ be a field, $G$ an affine group scheme over $k$, $X$ a connected scheme over $k$,
$f:Y\rightarrow X$ a $G$-torsor, $\mathcal{C}_X\subset \Vect(X)$
the Tannakian category on $X$. Suppose there exists a filtered family $\{W_i\}_{i\in I}$ of
finite-dimensional $G$-subrepresentations of the regular
representation $k[G]$ such that
$k[G]=\varinjlim_{i\in I}W_i$.
If $\eta_X^Y(W_i)\in\mathcal{C}_X$
for every $i \in I$, then $\eta_X^Y\bigl(\Rep_k^f(G)\bigr)\subseteq\mathcal{C}_X$. In particular, if $G$ is a finite group scheme over $k$ and $f_*\cO_Y \in \mathcal{C}_X$, then $\eta_X^Y\bigl(\Rep_k^f(G)\bigr)\subseteq\mathcal{C}_X$.
\end{Proposition}

\begin{proof}
Let $V\in\Rep_k^f(G)$. By
Lemma~\ref{regularrepresentation}, there exists a
$G$-equivariant monomorphism
$\varphi:V\hookrightarrow k[G]^{\oplus n}$
for some $n>0$. Since $V$ is finite-dimensional and
$k[G]=\varinjlim_{i \in I} W_i$, there exists
$\alpha\in I$ such that $\varphi$ factors as a monomorphism
$\varphi_\alpha:V\hookrightarrow W_\alpha^{\oplus n} \subseteq k[G]^{\oplus n}.$
Let $Q:=\coker(\varphi_\alpha)$
in $\Rep_k^f(G)$. Applying
Lemma~\ref{regularrepresentation} to $Q$, we obtain a
$G$-equivariant monomorphism
$\psi:Q\hookrightarrow k[G]^{\oplus m}$
for some $m>0$. There exists
$\beta\in I$ such that
 $\psi$ factors through a $G$-equivariant monomorphism
$\psi_\beta:Q\hookrightarrow W_\beta^{\oplus m} \subseteq k[G]^{\oplus m}$.
Then there exists an exact sequence of $G$-modules $ 0\rightarrow V
      \xrightarrow{\varphi_\alpha}W_\alpha^{\oplus n}
      \xrightarrow{u}W_\beta^{\oplus m}$.
Since $\eta_X^Y$ is exact, we obtain an exact sequence
\[
    0\rightarrow\eta_X^Y(V)
    \xrightarrow{\eta^Y_X(\varphi_\alpha)}\eta_X^Y(W_\alpha)^{\oplus n}
      \xrightarrow{\eta_X^Y(u)}
      \eta_X^Y(W_\beta)^{\oplus m}.
\]
Hence 
$   \eta_X^Y(V)
    \simeq
    \ker\bigl(\eta_X^Y(u)\bigr)
    \in\mathcal{C}_X$.
This proves that
$\eta_X^Y\bigl(\Rep_k^f(G)\bigr)\subseteq\mathcal{C}_X$.

If $G$ is a finite group scheme, take $W_\alpha = k[G]$, then $\eta^Y_X(k[G])= f_*\cO_Y$, we have $\eta_X^Y\bigl(\Rep_k^f(G)\bigr)\subseteq\mathcal{C}_X$.
\end{proof}

\begin{Theorem}\label{thm:general-exact-sequence}
Let $k$ be a field, $G$ an affine group scheme over $k$, $X$ a connected scheme proper over $k$,
$f:Y\rightarrow X$ a $G$-torsor, $\mathcal{C}_X\subset \Vect(X)$ and $\mathcal{C}_Y \subset \Vect(Y)$
 the Tannakian categories on $X$ and $Y$ respectively, such that $f^*\mathcal{C}_X\subseteq \mathcal{C}_Y$, $y\in Y(k)$ a rational point lying over $x \in X(k)$. 
Then:
\begin{enumerate}
 \item If $\eta_X^Y\bigl(\Rep^f_k(G)\bigr)\subseteq \mathcal{C}_X$, then the tensor functor $\eta_X^Y:\Rep^f_k(G) \to \mathcal{C}_X$ induces a faithfully flat homomorphism $ 
 \pi(\mathcal{C}_X,x) \to G$.
\item If $\eta_X^Y\bigl(\Rep^f_k(G)\bigr)\subseteq \mathcal{C}_X$, then $f^*:\mathcal{C}_X\to \mathcal{C}_Y$ is observable if and only if the natural sequence of affine group schemes over $k$ is exact:
\[
\pi(\mathcal{C}_Y, y)\rightarrow \pi(\mathcal{C}_X, x)\rightarrow G \rightarrow 1.
\]
\item $f_*\mathcal{C}_Y \subseteq \mathcal{C}_X$ if and only if $G$ is a finite $k$-group scheme, $\eta_X^Y\bigl(\Rep^f_k(G)\bigr)\subseteq \mathcal{C}_X$ and the natural sequence of affine group schemes over $k$ is exact:
\[
        1\rightarrow \pi(\mathcal{C}_Y,y)
        \rightarrow \pi(\mathcal{C}_X,x)
        \rightarrow G
        \rightarrow 1.
\]  
\end{enumerate}
\end{Theorem}

\begin{proof}
Consider the Cartesian diagram
\[
\begin{tikzcd}
G \times_k Y \arrow[r,"\alpha"] \arrow[d,"\operatorname{pr}_2"'] & Y \arrow[d,"f"]\\
Y \arrow[r,"f"'] & X .
\end{tikzcd}
\]
where $\alpha:G \times_k Y \rightarrow Y$ is the $G$-action and 
$\operatorname{pr}_2:G \times_k Y \rightarrow Y$ is the second projection.

(1)  Suppose $\eta_X^Y\bigl(\Rep^f_k(G)\bigr)\subseteq \mathcal{C}_X$. Since $\mathcal{C}_Y$ is a neutral Tannakian full subcategory of
$\Vect(Y)$, we have $H^0(Y,\cO_Y)=\End_{\mathcal{C}_Y}(\mathbf 1)=k$. By Proposition~\ref{fully faith}, $\eta_X^Y: \Rep^f_k(G) \rightarrow \mathcal{C}_X$ is fully faithful.

Denote $\mathcal{C}_{Y/X}:= \bigl\{\mathcal V\in \mathcal{C}_X \bigm| f^*\mathcal V \text{ is trivial in }\mathcal{C}_Y \bigr\} $. Let $\mathcal V \in\mathcal{C}_{Y/X}$. 
Since $f^*\mathcal V$ is trivial in $\mathcal{C}_Y$, we have 
$f^*\mathcal V \cong H^0(Y, f^*\mathcal V) \otimes_k \cO_Y$. The $G$-linearization 
$\alpha^*f^*\mathcal V\cong \operatorname{pr}_2^*f^*\mathcal V$ induces a $G$-action on $ H^0(Y, f^*\mathcal V)$,
hence we obtain a functor
\[
    \tau_X^Y:\mathcal{C}_{Y/X}\rightarrow\Rep_k^f(G),
    \quad
    \mathcal V\mapsto H^0(Y,f^*\mathcal V).
\]
 For any  $V \in \Rep^f_k(G)$,  we obtain \[
\tau_X^Y\eta_X^Y(V) \cong \tau_X^Y((\cO_Y\otimes_k V)^G) \cong H^0(Y,\cO_Y \otimes_k V) \cong V.
\]
Conversely, let $\mathcal V \in \mathcal{C}_{Y/X}$. Since $f^*\mathcal V$ is trivial in $\mathcal{C}_Y$ and 
\[
\eta_X^Y\tau_X^Y(\mathcal V) \cong \eta_X^Y(H^0(Y,f^*\mathcal V)) \cong (\cO_Y \otimes_k H^0(Y,f^*\mathcal V))^G  \cong (f^*\mathcal V)^G \cong \mathcal V.
\]
Hence $\eta^Y_X:\Rep^f_k(G) \rightarrow \mathcal{C}_{Y/X}$ defines an equivalence of categories with inverse functor $\tau_X^Y:\mathcal{C}_{Y/X}\rightarrow\Rep_k^f(G)$.

Let $\mathcal V \in \mathcal{C}_{Y/X}$ and $\mathcal W$ a subobject of $\mathcal V$ in $\mathcal{C}_X$, then $f^*\mathcal W$ is also a subobject of $f^*\mathcal V$ in $\mathcal{C}_Y$. Since every subobject of a trivial object in $\mathcal{C}_Y$ is also trivial, we have $f^*\mathcal W$ is trivial in $\mathcal{C}_Y$, i.e. $\mathcal W \in \mathcal{C}_{Y/X}$. Thus $\mathcal{C}_{Y/X}$ is closed under taking subobjects in $\mathcal{C}_X$. Hence $\mathcal{C}_{Y/X}$ is a Tannakian subcategory of $\mathcal{C}_X$. This implies the natural homomorphism $ 
 \pi(\mathcal{C}_X,x) \to G$ induced by the Tannakian dual of the  functor $\eta_X^Y:\Rep^f_k(G) \to \mathcal{C}_X$ is faithfully flat. 

(2) Since $\eta_X^Y\bigl(\Rep^f_k(G)\bigr)\subseteq \mathcal{C}_X$,  the homomorphism $\pi(\mathcal{C}_X, x) \to G$ is faithfully flat by (1).
 
Suppose $f^*:\mathcal{C}_X\to \mathcal{C}_Y$ is observable. We check the conditions~${\rm (a)}$ to ${\rm (c)}$ of (3) in Theorem~\ref{Thm1}.
Let $\mathcal V \in \mathcal{C}_X$ such that $f^*\mathcal V$ is trivial in $\mathcal{C}_Y$, then there exists  $ V \in \Rep^f_k(G)$ such that $\eta^Y_X(V) \cong \mathcal V$. Thus   condition ${\rm (a)}$ holds.
Let  $\mathcal V\in \mathcal{C}_X$. 
Then $\mathcal O_Y \otimes_k H^0(Y, f^*\mathcal V)$ is the maximal trivial subobject of $f^*\mathcal V$ in $\mathcal{C}_Y$. 
The isomorphism
$\lambda: \alpha^*f^*\mathcal V\cong \operatorname{pr}_2^*f^*\mathcal V$ induces the canonical
$G$-linearization on $f^*\mathcal V$.
Moreover, the morphism
\[
H^0(Y,f^*\mathcal V)
\rightarrow
H^0(G\times_kY,\alpha^*f^*\mathcal V)
\xrightarrow{H^0(\lambda)}
H^0(G\times_kY,\operatorname{pr}_2^*f^*\mathcal V)
\cong
k[G]\otimes_k H^0(Y,f^*\mathcal V)
\]
induces a $G$-representation on $ H^0(Y, f^*\mathcal V)$. Endow
$\mathcal O_Y\otimes_k H^0(Y,f^*\mathcal V)$
with the corresponding $G$-linearization, then $ \mathcal O_Y \otimes_k H^0(Y, f^*\mathcal V) \hookrightarrow f^*\mathcal V$ is $G$-equivariant.
By fpqc descent, 
it descends to a monomorphism
$ 
    \mathcal V_0 := \eta_X^Y\bigl(H^0(Y,f^*\mathcal V)\bigr)
    \hookrightarrow
    \mathcal V$ in $\mathcal{C}_X$.
We obtain
$f^*\mathcal V_0 \cong \mathcal O_Y \otimes_k H^0(Y, f^*\mathcal V)$. Thus condition~{\rm (b)} holds.
Finally, the condition ${\rm (c)}$ follows directly from the definition of observability and Lemma~\ref{observable}.
This proves that
$\pi(\mathcal{C}_Y, y)\rightarrow \pi(\mathcal{C}_X, x)\rightarrow G \rightarrow 1$
is exact.

Conversely, suppose the natural sequence $\pi(\mathcal{C}_Y, y)\rightarrow \pi(\mathcal{C}_X, x)\rightarrow G \rightarrow 1$ is  exact. Then $f^*:\mathcal{C}_X\to \mathcal{C}_Y$ is observable by Theorem~\ref{Thm1}~(3) and Lemma~\ref{observable}.

(3) Suppose $f_*\mathcal{C}_Y \subseteq \mathcal{C}_X$. Since $f: Y \rightarrow X$ is a  $G$-torsor, we have the canonical isomorphism $f^*f_*\cO_Y \simeq
\cO_Y \otimes_k k[G]$. Since $f_*\cO_Y\in \mathcal{C}_X$, it follows that $\dim_k k[G]<\infty$. Hence $G$ is a finite $k$-group scheme.
By Proposition~\ref{regular-representation-inclusion}, we have 
$\eta_X^Y(\Rep^f_k(G))\subseteq \mathcal{C}_X$.
Let $\mathcal V \in \mathcal{C}_X$ and 
$E'\hookrightarrow f^*\mathcal V$
be a subobject in $\mathcal{C}_Y$. Since $f$ is affine, the adjunction counit
$f^*f_*E'\twoheadrightarrow E'$
is an epimorphism. Since
$f_*E'\in\mathcal{C}_X$, 
we conclude that $f^*:\mathcal{C}_X \to\mathcal{C}_Y$ is observable by  Lemma~\ref{observable}.
Let $E\in \mathcal{C}_Y$. By assumption, we have $f_*E\in \mathcal{C}_X$. Since $f$ is affine, the adjunction counit
$f^*f_*E\twoheadrightarrow E$
is an epimorphism in $\mathcal{C}_Y$.  Theorem~\ref{Thm1}~(2) implies that
$\pi(\mathcal{C}_Y,y)\rightarrow \pi(\mathcal{C}_X,x)$
is a closed immersion. Combining with (2), we have the exact sequence 
$      1\rightarrow \pi(\mathcal{C}_Y,y)
        \rightarrow \pi(\mathcal{C}_X,x)
        \rightarrow G
        \rightarrow 1$.

Conversely, assume that $G$ is a finite $k$-group scheme, 
$\eta_X^Y(\Rep_k^f(G))\subseteq\mathcal C_X$ and the natural sequence
$1\rightarrow\pi(\mathcal{C}_Y,y)
\rightarrow\pi(\mathcal{C}_X,x)
\rightarrow G
\rightarrow1$
is exact.
Let $E\in\mathcal C_Y$. Since
$\pi(\mathcal{C}_Y,y)\to\pi(\mathcal{C}_X,x)$ is a closed immersion,
$E$ is a subquotient of $f^*F$ for some $F\in\mathcal{C}_X$ by Theorem~\ref{Thm1}~(2). By (2), $f^*$ is observable,
then by Lemma~\ref{observable}, we obtain an epimorphism
$f^*F_0\twoheadrightarrow E$
for some $F_0\in\mathcal C_X$.
Applying the same argument to its kernel, we obtain
$F_1\in\mathcal C_X$ and an exact sequence
$f^*F_1\rightarrow f^*F_0\rightarrow E\rightarrow 0$.
Since $G$ is a finite $k$-group scheme, we have $k[G]$ is finite-dimensional and $f$ is finite, then $f_*\cO_Y\simeq\eta_X^Y(k[G])\in\mathcal{C}_X$ and $f_*$ is exact. By the projection formula,
$F_1\otimes_{\cO_X} f_*\mathcal O_Y
\rightarrow
F_0\otimes_{\cO_X} f_*\mathcal O_Y
\rightarrow
f_*E
\rightarrow0$
is an exact sequence in $\mathcal{C}_X$. 
Since $\mathcal{C}_X$ is abelian, we have $f_*E \in \mathcal{C}_X$. Hence
$f_*\mathcal C_Y\subseteq\mathcal C_X$.
\end{proof}
\begin{Proposition}\label{saturation observable}
Let $k$ be a field, $f:Y \rightarrow X$ a morphism between connected schemes over $k$, $\mathcal{C}_X\subset \Vect(X)$ and $\mathcal{C}_Y \subset \Vect(Y)$
 the Tannakian categories on $X$ and $Y$ respectively, $f^*\mathcal{C}_X \subseteq \mathcal{C}_Y$. Then $f^*: \mathcal{C}_X \rightarrow \mathcal{C}_Y$ is observable if and only if  $f^*: \overline{\mathcal{C}}_X \to \overline{\mathcal{C}}_Y$ is observable.
\end{Proposition}
\begin{proof}
Suppose $f^*: \mathcal{C}_X \rightarrow \mathcal{C}_Y$ is observable.
Let $E \in \overline{\mathcal{C}}_X$ and  $L\in \overline{\mathcal{C}}_Y$ a rank-one subobject of $f^*E$, then $L \in \mathcal{C}_Y$. Choose a filtration
\[
0=E_0\hookrightarrow E_1\hookrightarrow\cdots\hookrightarrow E_m=E,
\qquad E_i/E_{i-1}\in\mathcal{C}_X.
\]
Since pullback is exact, this induces a filtration of $f^*E$. Consider the
subobjects $L\cap f^*E_i\subseteq L$ in $\overline{\mathcal{C}}_Y$. Since $L$ has
rank one, each is either $0$ or $L$. Let $i$ be the smallest index such that
$L\cap f^*E_i=L$. Then $L\cap f^*E_{i-1}=0$, and the composite
$L\hookrightarrow f^*E_i\twoheadrightarrow f^*(E_i/E_{i-1})$
is a monomorphism in $\mathcal{C}_Y$. By observability of
$f^*:\mathcal{C}_X\to\mathcal{C}_Y$, there exist $n>0$ and $F\in\mathcal{C}_X$ such that
$(L^{\otimes n})^\vee\hookrightarrow f^*F$.
Thus $f^*:\overline{\mathcal{C}}_X\to\overline{\mathcal{C}}_Y$ is observable.

Conversely, suppose that
$f^*:\overline{\mathcal{C}}_X\to\overline{\mathcal{C}}_Y$ is observable. Let
$E\in\mathcal{C}_X$ and let $L\hookrightarrow f^*E$ be a rank-one subobject in
$\mathcal{C}_Y$. There exist $n>0$ and $F\in\overline{\mathcal{C}}_X$ such that
$(L^{\otimes n})^\vee\hookrightarrow f^*F$.
Choose a filtration
\[
0=F_0\hookrightarrow F_1\hookrightarrow\cdots\hookrightarrow F_m=F,
\quad F_i/F_{i-1}\in\mathcal{C}_X.
\]
Applying the same argument to the rank-one subobject
$(L^{\otimes n})^\vee\hookrightarrow f^*F$, we find an index $i$ for which
the induced morphism
$(L^{\otimes n})^\vee\hookrightarrow f^*(F_i/F_{i-1})$
is a monomorphism. Since $F_i/F_{i-1}\in\mathcal{C}_X$, this proves that
$f^*:\mathcal{C}_X\to\mathcal{C}_Y$ is observable.

\end{proof}

\begin{Corollary}\label{thm:saturation-exact-sequence}
Let $k$ be a field, $G$ an affine group scheme over $k$, $X$ a connected scheme proper over $k$,
$f:Y\rightarrow X$ a $G$-torsor, $\mathcal{C}_X\subset \Vect(X)$ and $\mathcal{C}_Y \subset \Vect(Y)$
 the Tannakian categories on $X$ and $Y$ respectively, such that $f^*\mathcal{C}_X\subseteq \mathcal{C}_Y$, $y\in Y(k)$ a rational point lying over $x \in X(k)$. 
Then:
\begin{enumerate}
\item If $\eta_X^Y\bigl(\Rep^f_k(G)\bigr)\subseteq \mathcal{C}_X$, then the following are equivalent:
\begin{enumerate}
    \item the natural sequence 
$
\pi(\mathcal{C}_Y, y)\rightarrow  \pi(\mathcal{C}_X, x)\rightarrow  G \rightarrow 1
$
of affine group schemes over $k$  is exact.
\item the natural sequence  
$\pi(\overline{\mathcal{C}}_Y, y)\rightarrow \pi(\overline{\mathcal{C}}_X, x)\rightarrow G \rightarrow 1$
of affine group schemes over $k$ 
 is exact.
\end{enumerate}
\item If $f_*\mathcal{C}_Y \subseteq \mathcal{C}_X$,   then the natural sequence 
$1\rightarrow         
\pi(\overline{\mathcal{C}}_Y,y)
        \rightarrow \pi(\overline{\mathcal{C}}_X,x)
        \rightarrow G
        \rightarrow 1$
of affine group schemes over $k$ is exact.
\end{enumerate}
\end{Corollary}
\begin{proof}
(1) $(a)\Rightarrow (b)$. If the natural sequence  
$\pi(\mathcal{C}_Y, y)\rightarrow \pi(\mathcal{C}_X, x)\rightarrow G \rightarrow 1$ is exact, by Theorem~\ref{thm:general-exact-sequence} (2), $f^*:\mathcal{C}_X \rightarrow \mathcal{C}_Y $ is observable. By Proposition~\ref{saturation observable}, $f^*: \overline{\mathcal{C}}_X \to \overline{\mathcal{C}}_Y$ is also observable. Hence  the natural sequence  
$\pi(\overline{\mathcal{C}}_Y, y)\rightarrow \pi(\overline{\mathcal{C}}_X, x)\rightarrow G \rightarrow 1$
of affine group schemes over $k$ 
is exact by Theorem~\ref{thm:general-exact-sequence} (2).

$(b)\Rightarrow (a)$. If the natural sequence  
$\pi(\overline{\mathcal{C}}_Y, y)\rightarrow \pi(\overline{\mathcal{C}}_X, x)\rightarrow G \rightarrow 1$ is exact, by Theorem~\ref{thm:general-exact-sequence} (2), $f^*:\overline{\mathcal{C}}_X \rightarrow \overline{\mathcal{C}}_Y $ is observable. By Proposition~\ref{saturation observable}, $f^*: \mathcal{C}_X \to \mathcal{C}_Y$ is also observable. Hence the natural sequence  
$\pi(\mathcal{C}_Y, y)\rightarrow \pi(\mathcal{C}_X, x)\rightarrow G \rightarrow 1$
of affine group schemes over $k$ 
is exact  by Theorem~\ref{thm:general-exact-sequence} (2).

(2) Suppose $f_*\mathcal{C}_Y \subseteq \mathcal{C}_X$. Since $f_*$ is exact,  we have $f_*\overline{\mathcal{C}}_Y \subseteq \overline{\mathcal{C}}_X$. By Theorem~\ref{thm:general-exact-sequence} (3), the natural sequence 
$1\rightarrow         
\pi(\overline{\mathcal{C}}_Y,y)
        \rightarrow \pi(\overline{\mathcal{C}}_X,x)
        \rightarrow G
        \rightarrow 1$
is exact.
\end{proof}

\section{Fundamental group schemes of torsors}
\begin{Definition}
Let $k$ be a field, $X$ a geometrically reduced connected scheme proper over $k$, $x\in X(k)$, $E$ a vector bundle on $X$ of rank $r$. If $k$ is of positive characteristic, then $F_X:X\rightarrow X$ is the absolute Frobenius morphism. Then $E$ is said to be
\begin{itemize}
\item \textit{numerically flat}, if both $E$ and $E^\vee$ are nef.
\item \textit{Nori semistable}, if for any smooth projective curve $f:C\rightarrow X$, $f^*E$ is semistable of degree 0.
\item \textit{finite}, if there exist $f(t)\neq g(t)\in\mathbb{N}[t]$ such that $f(E)\cong g(E)$, where
\[
   h(E):=\bigoplus\limits_{i=0}^m(E^{\otimes i})^{\oplus n_i}\text{ for any }h(t)=\sum\limits_{i=0}^mn_it^i\in\mathbb{N}[t].
\]
\item \textit{essentially finite}, if there exists $E_1\hookrightarrow E_2\hookrightarrow F\in\Vect(X)$ such that $E\cong E_2/E_1$, where $E_1,E_2$ are numerically flat and $F$ is finite.
\item \textit{Frobenius finite}, if there exist $f(t)\neq g(t)\in\mathbb{N}[t]$ such that $\tilde{f}(E)\cong \tilde{g}(E)$, where 
\[
\tilde{h}(E):=\bigoplus\limits_{i=0}^{m}((F_X^i)^*E)^{\oplus n_i} \text{ for any }h(t)=\sum\limits_{i=0}^mn_it^i\in\mathbb{N}[t].
\]
\item \textit{essentially Frobenius finite}, if there exists $E_1\hookrightarrow E_2\hookrightarrow F\in\Vect(X)$ such that $E\cong E_2/E_1$, where $E_1,E_2$ are numerically flat and $F$ is Frobenius finite.
\item \textit{Frobenius trivial}, if there exists a positive integer $n$ such that $F_X^{n*} E\cong \cO_X^{\oplus r}$.
\item \textit{\'etale trivializable}, if there exists a finite \'etale covering $\phi: P\rightarrow X$ such that $\phi^* E\cong \cO_P^{\oplus r}$.
\item \textit{unipotent}, if there is a filtration
$0 = E_0 \hookrightarrow E_1 \hookrightarrow \cdots \hookrightarrow E_n = E$
such that
$E_i/E_{i-1} \cong \cO_X$
for any $i$.
\end{itemize}
We have the following Tannakian categories:
    \begin{itemize}
        \item $\mathcal{C}^{NF}(X)$: objects consist of numerically flat bundles on $X$.
        \item $\mathcal{C}^{N}(X)$: objects consist of essentially finite bundles on $X$.
        \item $\mathcal{C}^{F}(X)$: objects consist of essentially Frobenius finite bundles on $X$.
        \item $\mathcal{C}^{Loc}(X)$: objects consist of Frobenius trivial bundles on $X$.
        \item $\mathcal{C}^{\acute{e}t}(X)$: objects consist of \'etale trivializable bundles on $X$.
        \item $\mathcal{C}^{uni}(X)$: objects consist of unipotent bundles on $X$.
    \end{itemize}
We have the following Tannaka group schemes:
\begin{itemize}
        \item $\pi^{S}(X,x):=\pi(\mathcal{C}^{NF}(X),x)$, called the \textit{S-fundamental group scheme}.
        \item $\pi^{N}(X,x):=\pi(\mathcal{C}^{N}(X),x)$, called the \textit{Nori fundamental group scheme}.
        \item $\pi^{EN}(X,x):=\pi(\overline{\mathcal{C}^{N}(X)},x)$, called the \textit{extended Nori fundamental group scheme}.
        \item $\pi^{F}(X,x):=\pi(\mathcal{C}^{F}(X),x)$, called the \textit{F-fundamental group scheme}.
        \item $\pi^{EF}(X,x):=\pi(\overline{\mathcal{C}^{F}(X)},x)$, called the \textit{extended F-fundamental group scheme}.
        \item $\pi^{Loc}(X,x):=\pi(\mathcal{C}^{Loc}(X),x)$, called the \textit{local fundamental group scheme}.
        \item $\pi^{ELoc}(X,x):=\pi(\overline{\mathcal{C}^{Loc}(X)},x)$, called the \textit{extended local fundamental group scheme}.
        \item $\pi^{\acute{e}t}(X,x):=\pi(\mathcal{C}^{\acute{e}t}(X),x)$, called the \textit{\'etale fundamental group scheme}.
        \item $\pi^{E\acute{e}t}(X,x):=\pi(\overline{\mathcal{C}^{\acute{e}t}(X)},x)$, called the \textit{extended \'etale fundamental group scheme}.
        \item $\pi^{uni}(X,x):=\pi(\mathcal{C}^{uni}(X),x)$, called the \textit{unipotent fundamental group scheme}.
    \end{itemize}
\end{Definition}

\begin{Remark}
We have  $\mathcal{C}^{\acute{e}t}(X) \subseteq \mathcal{C}^{N}(X) \subseteq \mathcal{C}^{NF}(X)$ and $\mathcal{C}^{Loc}(X) \subseteq \mathcal{C}^{F}(X) \subseteq \mathcal{C}^{N}(X) \subseteq \mathcal{C}^{NF}(X)$.
\end{Remark}

\begin{Lemma}[{\cite[Proposition~1.4.3]{CLM22}}]\label{surj nef}
 Let $k$ be a field, $X$ and $Y$ proper algebraic spaces over $k$, $f: Y \to X$ a surjective proper morphism, let $E \in \Vect(X) $. 
 If $f^*E$ is nef, then $E$ is nef.
\end{Lemma}

\begin{Lemma}[{\cite[Proposition~1.4.4]{CLM22}}]\label{nef}
    Let $k$ be a field, $X$ a proper algebraic space over $k$, $E \in \Vect(X) $. Then $E$ is nef if and only if  the pullback $E\otimes_k k'$ on $X_{k'}$ is nef for every field extension $k \subset k'$.
\end{Lemma} 

\begin{Proposition}\label{general observable}
Let $k$ be a field, $f: Y\rightarrow X$ a morphism between 
geometrically reduced connected schemes proper over $k$, $\mathcal{C}_X\subset \Vect(X)$ and
$\mathcal{C}_Y\subset \Vect(Y)$ the Tannakian categories on $X$ and $Y$ respectively, such that $f^*\mathcal{C}_X \subseteq \mathcal{C}_Y \subseteq \mathcal{C}^{EN}(Y)$. Then $f^*: \mathcal{C}_X \to \mathcal{C}_Y$ is observable.
In particular,
$f^*:\mathcal{C}^{*}(X)\rightarrow\mathcal{C}^{*}(Y)$ is observable for
$*\in
\{N,EN,F,EF,Loc,ELoc,\acute{e}t,E\acute{e}t,uni\}$.
\end{Proposition}
\begin{proof}
For any $E \in \mathcal{C}_X$ and any rank-one subobject $L \hookrightarrow f^*E$ in $\mathcal{C}_Y \subseteq \mathcal{C}^{EN}(Y)$, we have $L \in \mathcal{C}^N(Y)$.  
Then $L^{\otimes n}\cong \cO_Y$ for some integer $n>0$. Since $ \cO_X \in \mathcal{C}_X$, we have 
$ (L^{\otimes n})^\vee
        \cong \cO_Y \cong
    f^*\cO_X$.
 Hence $f^*:\mathcal{C}_X\rightarrow\mathcal{C}_Y$ is observable.
\end{proof}

\begin{Proposition}\label{NF observable}
Let $k$ be a field,  $X$ and $Y$  geometrically reduced connected proper $k$-schemes, and 
$f:Y\rightarrow X$  a $G$-torsor, where $G$ is a finite group scheme over $k$.
Then $f_*\bigl(\mathcal{C}^{NF}(Y)\bigr) \subseteq \mathcal{C}^{NF}(X)$.
\end{Proposition}
\begin{proof}
Let $E\in\mathcal{C}^{NF}(Y)$. Since $Y$ is a  $G$-torsor over $X$, then $f_*E$ is a vector bundle. Denote by
$\alpha:G\times_k Y \rightarrow Y$
the action and by
$p:G\times_k Y\rightarrow Y$
the second projection. The isomorphism
$G\times_k Y\xrightarrow{\sim}Y\times_XY$
and flat base change give
$f^*f_*E \cong p_{*}\alpha^*E$.
Let $\overline{k}$ be an algebraic closure of $k$, and write
\[
    G_{\overline{k}}=\coprod_{i=1}^rG_i,
    \qquad
    G_i=\operatorname{Spec}(A_i),
\]
where each $A_i$ is a local Artinian $\overline{k}$-algebra with maximal ideal $\mathfrak m_i$ and residue field $\overline{k}$. Let
\[
    \alpha_i:G_i \times_{\overline{k}} Y_{\overline{k}}\rightarrow Y_{\overline{k}},
    \quad
    p_i :G_i\times_{\overline{k}} Y_{\overline{k}}\rightarrow Y_{\overline{k}}
\]
be the restrictions of the action and the second projection. Let $g_i\in G_i(\overline{k})$ be the closed point,
$\iota_i:Y_{\overline{k}}
    \cong \{g_i\}\times_{\overline{k}} Y_{\overline{k}}
    \hookrightarrow G_i \times_{\overline{k}} Y_{\overline{k}}$
the closed immersion, and $t_{g_i}:Y_{\overline{k}}\to Y_{\overline{k}}$ denote translation by $g_i$. Thus
$\alpha_i\circ\iota_i=t_{g_i}$ and $p_i\circ\iota_i=\operatorname{Id}_{Y_{\overline{k}}}$.

Let $\mathcal J_i:=\mathfrak m_i\mathcal O_{G_i\times_{\overline{k}}Y_{\overline{k}}}$ be the ideal sheaf induced by the projection $q_i: G_i \times_{\overline{k}} Y_{\overline{k}} \rightarrow G_i$. Choose $N_i>0$ such that $\mathcal J_i^{N_i}=0$, we have a finite filtration
\[
    \alpha_i^*E_{\overline{k}} = \mathcal J_i^0 \alpha_i^*E_{\overline{k}}\supseteq\mathcal J_i^1 \alpha_i^*E_{\overline{k}} \supseteq\cdots
    \supseteq \mathcal J_i^{N_i} \alpha_i^*E_{\overline{k}}=0.
\]
Since $p_i$ is finite, $p_{i*}$ is exact. Thus
\[
    p_{i*}\alpha_i^*E_{\overline{k}} =p_{i*}\bigl(\mathcal J_i^0 \alpha_i^*E_{\overline{k}}\bigr)\supseteq p_{i*}\bigl(\mathcal J_i^1 \alpha_i^*E_{\overline{k}}\bigr) \supseteq\cdots
    \supseteq  p_{i*}\bigl(\mathcal  J_i^{N_i} \alpha_i^*E_{\overline{k}}\bigr)=0
\]
is also a finite filtration. 
Set $\mathcal F_i^j:=p_{i*}\bigl(\mathcal J_i^j \alpha_i^*E_{\overline{k}}\bigr)$, $0\leq j\leq N_i$.
Since $\alpha_i^*E_{\overline{k}}$ is locally free, by the projection formula, its successive quotients satisfy
\[
\begin{aligned}
\mathcal F_i^j/\mathcal F_i^{j+1} 
&\cong p_{i*}\left( \mathcal J_i^j \alpha_i^*E_{\overline{k}} \big/ \mathcal J_i^{j+1}\alpha_i^*E_{\overline{k}} \right) 
\cong p_{i*}\left( \mathcal J_i^j \big/ \mathcal J_i^{j+1} \otimes_{\cO_{G_i \times Y_{\overline{k}}}} \alpha_i^*E_{\overline{k}}\right) \\
&\cong p_{i*}\left( (\iota_{i*}\cO_{Y_{\overline{k}}})^{\oplus d_{ij}} \otimes_{\cO_{G_i \times Y_{\overline{k}}}}\alpha_i^*E_{\overline{k}} \right) 
\cong p_{i*}\iota_{i*} \bigl(t_{g_i}^*E_{\overline{k}}\bigr)^{\oplus d_{ij}} 
\cong \bigl(t_{g_i}^*E_{\overline{k}}\bigr)^{\oplus d_{ij}},
\end{aligned}
\]
where $d_{ij}:=\dim_{\overline{k}}
    \mathfrak m_i^j/\mathfrak m_i^{j+1}$.
Each $t_{g_i}$ is an automorphism, so $t_{g_i}^*E_{\overline{k}}$ is numerically flat on $Y_{\overline{k}}$. Since numerically flat bundles are closed under finite direct sums and extensions, each $p_{i*}\alpha_i^*E_{\overline{k}}$ is numerically flat on $Y_{\overline{k}}$. Since finite pushforward commutes with field extension, we obtain
\[
    (f^*f_*E)_{\overline{k}}
    \cong
    \bigoplus_{i=1}^r p_{i*}\alpha_i^*E_{\overline{k}},
\]
which is numerically flat.

By Lemma~\ref{nef}, the vector bundle $f^*f_*E$ is numerically flat on $Y$. Applying Lemma~\ref{surj nef} to the surjective proper morphism $f$,  $f_*E$ is numerically flat on $X$. This proves that 
$f_*\bigl(\mathcal{C}^{NF}(Y)\bigr) \subseteq \mathcal{C}^{NF}(X)$.
\end{proof}

\begin{Proposition}\label{exact-certain-categories}
Let $k$ be a field, $X$ and $Y$ geometrically reduced connected schemes proper over $k$,
$f:Y\rightarrow X$ a $G$-torsor with $G$  a finite group scheme over $k$, and $y\in Y(k)$ a rational point lying over $x \in X(k)$.
Then for $*\in
\{S,N,EN\}$, the natural sequence of affine $k$-group schemes is exact:
\[
       1 \rightarrow
        \pi^*(Y,y)
        \rightarrow
        \pi^*(X,x)
        \rightarrow
        G
        \rightarrow 1.
\]
\end{Proposition}
\begin{proof}
Case $* = S$.
By Proposition~\ref{NF observable}, we have  $f_*\mathcal{C}^{NF}(Y) \subseteq \mathcal{C}^{NF}(X)$. Hence 
Theorem~\ref{thm:general-exact-sequence}~(3) implies that the natural sequence
$       1 \rightarrow
        \pi^S(Y,y)
        \rightarrow
        \pi^S(X,x)
        \rightarrow
        G
        \rightarrow 1$
is exact.

Case $*= N$. By \cite[Theorem~3.8]{AnEm11}, $f_*E \in \mathcal{C}^N(X)$ for any $E \in \mathcal{C}^{N}(Y)$. Hence Theorem~\ref{thm:general-exact-sequence}~(3) implies that the natural sequence
$       1 \rightarrow
        \pi^N(Y,y)
        \rightarrow
        \pi^N(X,x)
        \rightarrow
        G
        \rightarrow 1$
is exact.

Case $*= EN$. Since $f_*$ is exact, we obtain that $f_*\mathcal{C}^{EN}(Y) \subseteq \mathcal{C}^{EN}(X)$. Hence Theorem~\ref{thm:general-exact-sequence}~(3) implies that the natural sequence
$       1 \rightarrow
        \pi^{EN}(Y,y)
        \rightarrow
        \pi^{EN}(X,x)
        \rightarrow
        G
        \rightarrow 1$
is exact.
\end{proof}

\begin{Remark}
For the Nori fundamental group scheme, the corresponding exact sequence for pointed torsors under finite group schemes was proved by Antei--Emsalem \cite[Theorem~4.1]{AnEm11}. It also follows from the more general framework of Nori fundamental gerbes developed by Antei--Biswas--Emsalem--Tonini--Zhang \cite[Corollary~I]{ABETZ19}. For the $S$-fundamental group scheme, the exact sequence is contained in \cite[Theorem~IV]{ABETZ19}.
For the extended Nori fundamental group scheme, Adroja--Amrutiya proved the analogous exact sequence for finite \'etale Galois covers; see \cite{AdAm25}. Thus the $S$- and Nori cases of Theorem~\ref{exact-certain-categories} recover known results from the general criterion above, while the extended Nori case extends the finite \'etale statement to torsors under arbitrary finite group schemes through the saturation formalism.
\end{Remark}

\begin{Theorem}\label{finite-etale-exact}
Let $k$ be a field, $X$ and $Y$ geometrically reduced connected schemes proper over $k$,
$f:Y\rightarrow X$  a finite \'etale Galois morphism with
Galois group $G=\operatorname{Aut}_X(Y)$, $y\in Y(k)$ a rational point lying over $x \in X(k)$. Let $*\in
\{F,EF,\acute e t,E\acute e t\}$. Then 
the natural sequence of affine $k$-group schemes is exact:
\[
     1 \rightarrow
        \pi^*(Y,y)
        \rightarrow
        \pi^*(X,x)
        \rightarrow
        G
        \rightarrow 1.
\]
\end{Theorem}
\begin{proof}
We first prove that         $f_*\bigl(\mathcal{C}^*(Y)\bigr)\subseteq \mathcal{C}^*(X)$
for
$*\in \{F,\acute e t\}$. 

Case $* =F$.
If $E$ is Frobenius finite,
there exist distinct polynomials
$a(t),b(t)\in\mathbb N[t]$ such that
$\widetilde a(E)\cong \widetilde b(E)$. Since
$f$ is \'etale,  \cite[Lemma~13.2]{Mil80} shows that for every $ m\geq 0$ the square
\[
\begin{tikzcd}
Y \arrow[r,"F_Y^m"] \arrow[d,"f"']&
Y \arrow[d,"f"]\\
X \arrow[r,"F_X^m"']&
X
\end{tikzcd}
\]
is Cartesian. Since $f$ is affine, then quasi-coherent pushforward commutes with arbitrary base change. Hence
$F_X^{m*}f_*E\cong f_*F_Y^{m*}E$.
Since $f_*$ commutes with finite direct
sums and Frobenius pullback, we obtain
$       \widetilde a(f_*E)
        \cong
        f_*\widetilde a(E)
        \cong
        f_*\widetilde b(E)
        \cong
        \widetilde b(f_*E)$.
Thus $f_*E$ is Frobenius finite.

Let $E\in\mathcal{C}^F(Y)$. By definition, there exist $E_1\hookrightarrow E_2\hookrightarrow F$ such that $E\cong E_2/E_1$, where $E_1$ and $E_2\in \mathcal{C}^{NF}(Y)$ and $F$ is a Frobenius finite bundle on $Y$. Since $f_*$ is exact, we have $f_*E_1\hookrightarrow f_*E_2\hookrightarrow f_*F$, where $f_*E_1$ and $f_*E_2\in\mathcal{C}^{NF}(X)$ by Proposition~\ref{NF observable}, and $f_*F$ is a Frobenius finite bundle on $X$. Thus $f_*E\cong f_*E_2/f_*E_1\in\mathcal{C}^F(X)$. Hence $f_*\bigl(\mathcal{C}^F(Y)\bigr)\subseteq \mathcal{C}^F(X)$.

Case $* =\acute e t$.  Let $E \in \mathcal{C}^{\acute e t}(Y)$,  
there exists a finite \'etale covering $u: P\rightarrow Y$
such that  $u^*E$ is trivial. For each $g\in G$, let $g:Y\to Y$ denote the corresponding deck transformation.  We have $f^*f_*E \cong \bigoplus_{g\in G}g^*E$, each $g^*E$ is trivializable by a finite \'etale covering $u_g: P^g:= P{}_{u}\times_{g}Y\rightarrow Y$. Since $G$ is finite, choose an enumeration $G=\{g_1, \cdots, g_n\}$.
For any $1\leq i \leq n$,
take  
\[
v \colon P^{g_1}\! {}_{u_{g_1}}{\times}_{u_{g_2}}\! P^{g_2}\! {}_{u_{g_2}}{\times}\! \cdots\! \times_{u_{g_{n-1}}}\! P^{g_{n-1}}\! {}_{u_{g_{n-1}}}{\times}_{u_{g_n}}\! P^{g_n} \xrightarrow{p_i} P^{g_i} \xrightarrow{u_{g_i}} Y,
\]
then $v$ is independent of $i$, $f\circ v$ is finite \'etale and trivializes $f_*E$, i.e. $f_*E \in \mathcal{C}^{\acute e t}(X)$. Hence we have $f_*\bigl(\mathcal{C}^{\acute{e}t}(Y)\bigr)\subseteq \mathcal{C}^{\acute{e}t}(X)$.

Hence the natural sequence 
$ 1 \rightarrow
        \pi^*(Y,y)
        \rightarrow
        \pi^*(X,x)
        \rightarrow
        G
        \rightarrow 1$
is exact for $*\in
\{F,\acute e t\}$ by Theorem~\ref{thm:general-exact-sequence} (3).
The cases $* \in \{ EF,E\acute{e}t\}$
follow directly from Corollary~\ref{thm:saturation-exact-sequence}.
\end{proof}

\begin{Example}\label{elliptic-counterexample}
Let $k$ be an algebraically closed field of characteristic $p>0$, $X$ an elliptic curve over $k$, and choose a prime number $\ell\neq p$. Consider the multiplication-by-$\ell$ morphism
$f=[\ell]\colon Y=X\rightarrow X$.
Since $\ell\neq p$, this is a finite \'etale Galois covering. Its Galois group is
$\Gamma=X[\ell](k)\cong (\mathbb Z/\ell\mathbb Z)^2$.
Clearly
$\cO_Y\in \mathcal{C}^{*}(Y)$ for $* \in\{Loc,ELoc,uni\}$.
We claim  that
$f_*\cO_Y\notin \mathcal{C}^{*}(X)$ for $* \in\{Loc,ELoc,uni\}$.

Indeed, since $\ell\neq p$,   Maschke’s Theorem implies that the group algebra $k[\Gamma]$ is semisimple. As $\Gamma$ is abelian and $k$ is algebraically closed, every irreducible representation of $\Gamma$ is one-dimensional. Hence
$k[\Gamma]\cong \bigoplus_{\chi\in \widehat{\Gamma}} k_\chi$,
where $\widehat{\Gamma}:=\operatorname{Hom}(\Gamma,k^\times)$ and $k_\chi$ denotes the one-dimensional representation corresponding to $\chi$.
We have $f_*\mathcal O_Y
\cong
\bigoplus_{\chi\in\widehat{\Gamma}} L_\chi$,
where
$L_\chi:=Y\times^\Gamma k_\chi$
is the line bundle on $X$ associated to the character $\chi$. For a nontrivial character $\chi\neq 1$, the line bundle $L_\chi$ is a nontrivial $\ell$-torsion line bundle.

We first show that $L_\chi\notin \mathcal{C}^{Loc}(X)$ for any $\chi\neq 1$. Since $L_\chi$ is a line bundle, for every $n>0$ one has
$F_X^{n*}L_\chi\cong L_\chi^{\otimes p^n}$.
Because $\ell\nmid p^n$, the nontrivial $\ell$-torsion line bundle $L_\chi^{\otimes p^n}$ is still nontrivial. Hence
$F_X^{n*}L_\chi\not\cong \mathcal O_X$
for all $n>0$. Therefore
$L_\chi\notin \mathcal{C}^{Loc}(X)$.
Since a rank-one unipotent bundle must be isomorphic to $\mathcal O_X$, we have
$L_\chi\notin \mathcal{C}^{uni}(X)$.

If $L_\chi \in \mathcal{C}^{ELoc}(X)$, it admits a filtration whose successive quotients belong to $\mathcal{C}^{Loc}(X)$. Since $L_\chi$ has rank one, such a filtration has only one nonzero quotient, so $L_\chi$ itself would lie in $\mathcal{C}^{Loc}(X)$, a contradiction. Hence
$L_\chi\notin \mathcal{C}^{ELoc}(X)$.
\end{Example}


\begin{thebibliography}{99}

\bibitem{AdAm25} P. Adroja and S. Amrutiya, \emph{On an extension of Nori and local fundamental group schemes}, Comm. Algebra {\bf 53} (2025), no.~10, 4241--4255.


\bibitem{AmBi10} S. Amrutiya, I. Biswas, \emph{On the F-fundamental group scheme}, Bull. Sci. Math. {\bf 134} (2010), no. 5, 461--474.

\bibitem{AnEm11}
M.~Antei and M.~Emsalem,
\emph{Galois closure of essentially finite morphisms},
J. Pure Appl. Algebra \textbf{215} (2011), no.~11, 2567--2585.

\bibitem{ABETZ19}
M.~Antei, I.~Biswas, M.~Emsalem, F.~Tonini, and L.~Zhang,
\emph{Nori fundamental gerbe of essentially finite covers and Galois closure
of towers of torsors},
Selecta Math. (N.S.) \textbf{25} (2019), no.~2, Paper No.~18, 37 pp.

\bibitem{AOV08}D. Abramovich, M.~C. Olsson and A. Vistoli, \emph{Tame stacks in positive characteristic}, Ann. Inst. Fourier (Grenoble) {\bf 58} (2008), no.~4, 1057--1091.

\bibitem{BHP15} I. Biswas, A. Hogadi and A.~J. Parameswaran, \emph{Fundamental group of a geometric invariant theoretic quotient}, Transform. Groups {\bf 20} (2015), no.~2, 367--379.

\bibitem{BHS21} I. Biswas, P. H. Hai and J.~P.~P. dos~Santos, \emph{On the fundamental group schemes of certain quotient varieties}, Tohoku Math. J. (2) {\bf 73} (2021), no.~4, 565--595.

\bibitem{CLM22} R. Cheng, C. Lian and T. Murayama, Projectivity of the moduli of curves, in {\it Stacks Project Expository Collection (SPEC)},  London Math. Soc. Lecture Note Ser., 480, Cambridge Univ. Press, Cambridge, 2022, 1--43.


\bibitem{DaEs22} M. D'Addezio, H. Esnault, \emph{On the universal extensions in Tannakian categories}, Int. Math. Res. Not. IMRN {\bf 2022}, no.~18, 14008--14033.

\bibitem{EHS07} H. Esnault, P. H. Hai, X. Sun, \emph{On Nori's fundamental group scheme}, Geometry and dynamics of groups and spaces, 377--398, Progr. Math., 265, Birkhäuser, Basel, 2007.

\bibitem{Gro60} A. Grothendieck, \emph{Revêtements étales et groupe fondamental}, Séminaire de Géométrie Algébrique du Bois-Marie, (SGA 1), 1960/61, Lecture Notes in Mathematics, Vol. 224, Springer-Verlag, Berlin, (1971).

\bibitem{Lan11} A. Langer, \emph{On the S-fundamental group scheme}, Ann. Inst. Fourier (Grenoble) {\bf 61} (2011), no.~5, 2077--2119 (2012).

\bibitem{Lan12} A. Langer, \emph{On the S-fundamental group scheme. II}, J. Inst. Math. Jussieu {\bf 11} (2012), no.~4, 835--854.

\bibitem{LiWang26} L.~Li and H. Wang,
\emph{The birational invariance Of fundamental group schemes}, \url{https://arxiv.org/abs/2604.23997v2}, 2026.

\bibitem{MeSu08} V.~B. Mehta and S. Subramanian, \emph{Some remarks on the local fundamental group scheme}, Proc. Indian Acad. Sci. Math. Sci. {\bf 118} (2008), no.~2, 207--211.

\bibitem{Mil80} J.~S. Milne, \emph{\'Etale cohomology}, Princeton Mathematical Series, No. 33, Princeton Univ. Press, Princeton, NJ, 1980.

\bibitem{Mil12} J. S. Milne, \emph{Basic Theory of Affine Group Schemes}, 2012, Available at \url{www.jmilne.org/math/}.


\bibitem{Nor76} M. V. Nori, \emph{On the representations of the fundamental group}, Compositio Math. {\bf 33} (1976), no.~1, 29--41.

\bibitem{Nor82} M. V. Nori, \emph{The fundamental group-scheme}, Proc. Indian Acad. Sci. Math. Sci. {\bf 91} (1982), no.~2, 73--122.

\bibitem{Ota17} S. Otabe, \emph{An extension of Nori fundamental group}, Comm. Algebra {\bf 45} (2017), no.~8, 3422--3448.


\bibitem{Vis05}A. Vistoli, Grothendieck topologies, fibered categories and descent theory, in {\it Fundamental algebraic geometry}, Math. Surveys Monogr., 123, Amer. Math. Soc., Providence, RI, 2005, 1--104. 

\bibitem{Zha14} L. Zhang, \emph{The homotopy sequence of the algebraic fundamental group},Int. Math. Res. Not. IMRN {\bf 2014}, no.~22, 6155--6174.

\end{thebibliography}
\end{document}